\documentclass{amsart}
\usepackage[paperwidth=7in, paperheight=10in, margin=.875in]{geometry}
\usepackage{amssymb,amsbsy,amsmath}
          \usepackage{graphicx}

            \newtheorem{theorem}{Theorem}[section]
          
          \newtheorem{lemma}[theorem]{Lemma}
          \newtheorem{corollary}[theorem]{Corollary}

        \newtheorem{remark}[theorem]{Remark}

\begin{document}
          \title[Nonlocal diffusion equation for image denoising and data analysis]{A new nonlocal nonlinear diffusion equation for image denoising and data analysis}

\markleft{G.~Aletti, M.~Moroni, G.~Naldi} 
         
          \author[G.~Aletti]{Giacomo Aletti}

\address{ADAMSS Center, Universit{\`a} degli studi di Milano,
          Via Saldini 50, 20133 Milano, Italy}
          \email{giacomo.aletti@unimi.it}
          
          \author[M.~Moroni]{Monica Moroni}
          \address{IIT Genova, Via Morego, 30,
         16163 Genova, Italy} \email{monica.moroni@iit.it}
         \author[G.~Naldi]{Giovanni Naldi}
         \address{ADAMSS Center, Universit{\`a} degli studi di Milano,
          Via Saldini 50, 20133 Milano, Italy}
          \email{giovanni.naldi@unimi.it}
          
          \thanks{This work was funded by ADAMSS Center, G. Naldi acknowledges the
support of the Hausdorff Research Institute for Mathematics during the Special
Trimester ''Mathematics of Signal Processing''.
G.Aletti and G.Naldi are members of ``Gruppo Nazionale per il Calcolo
  Scientifico (GNCS)'' of the Italian Institute ``Istituto Nazionale
  di Alta Matematica (INdAM)''}



          \begin{abstract}
In this paper we introduce and study a new  feature-preserving nonlinear anisotropic diffusion for denoising signals.
The proposed partial differential equation is based on a novel diffusivity  coefficient that uses a nonlocal
automatically  detected parameter related to the local bounded variation and the local oscillating pattern of the
noisy input signal. We provide a mathematical analysis of the existence of the solution  of our nonlinear and nonlocal
diffusion equation in the two dimensional case (images processing). Finally, we propose a numerical scheme
with some numerical experiments which demonstrate the effectiveness of the new method.
          \end{abstract}

         \maketitle

\section{Introduction}

Nonlinear partial differential equations (PDEs) can be used in the analysis and processing of digital images or image sequences, for example  to filter out the noise, to produce higher quality image,  to extract features and shapes (see e.g. \cite{AGLM93,AK2006, APT2006,GS2001,Wei98} and the References herein). Perhaps, the main application of PDEs based methods in this field is
smoothing and restoration of images. From the mathematical point of view, the input (grayscale) image can be moelled by a real function
$u_0(x)$, $u_0:\Omega\rightarrow\mathbb{R}$, where $\Omega \subset \mathbb{R}^d$, represents the spatial domain. Typically this domain $\Omega$ is rectangular and $d=1,~2,~3$. The function $u_0$ is considered as an initial data for a suitable evolution equation with
some kind of  boundary conditions (usually homogeneous Neumann boundary conditions).

The simplest  PDE method for smoothing images is to apply a linear diffusion process, the starting point is the simple observation that the so called Gauss function, with $\sigma >0$ and where
$|\cdot |$ is the Euclidean norm,
\[
G_\sigma (x) = \frac{1}{(2\pi \sigma^2 )^{d/2}}\, e^{-|x|^2/(2 \sigma^2 )}
\]
is related to the fundamental solution of the linear diffusion (heat) equation. Then, it has been possible to reinterpret the classical smoothing operation of the convolution of an image with $G_\sigma$, with a given standard deviation $\sigma$, by solving the linear diffusion equation for a corresponding time $t=\sigma^2 /2$ with initial condition given by the original image. For example, when $d=2$, it is a classic result that for any bounded, continuous. and integrable $u_0 (x)$, $x\in\mathbb{R}^2$, the linear diffusion equation on the whole space (here $\triangle$ denotes the Laplacian operator),
\[
\frac{\partial u}{\partial t} = \triangle u,\,\, u(x,0) =u_0(x)
\]
possesses the following solution
\[
u(x,t) =
\left\{
\begin{array}{l}
  u_0(x),\,\,\, t=0 \\
  \\
  (G_{\sqrt{2t}}*u_0)(x),\, t>0
\end{array}
\right.
\]
where the convolution product $(g*f)(x)$  between the function $f$ and $g$, is defined as
\[
(g*f)(x) =\int_{\mathbb{R}^2} g(x-y) f(y) dy.
\]
We point out that for different time (variance) $t$ we obtain different levels of smoothing: this
defines a \textit{scale-space} for the image \cite{Koe84,Wit83}. That is, we get copies of the image at different
scales. Note, of course, that any scale $t$ can be obtained from a scale $\tau$, where $\tau < t$,
as well as from the original images, this is usually denoted as the causality criteria for scale-spaces \cite{AGLM93}.
The solution of the above linear diffusion equation is unique, provided we restrict ourselves to functions satisfying
some suitable grow conditions. Moreover, it depends continuously on the initial image $u_0$, and it fulfils the maximum/minimum principle
\[
\inf_{x\in\mathbb{R}^2} u_0(x) \leq u(x,t) \leq \sup_{x\in\mathbb{R}^2} u_0(x)\,\,on\,\,\mathbb{R}^2\times [0,\infty ).
\]
For application in image processing we also need to consider appropriate boundary conditions: usually homogeneous Neumann conditions are used.

The flow produced by the linear diffusion equation is also denoted as isotropic diffusion, as it is diffusing
the information equally in all directions. Then, the gray values of the initial image will spread, and, in the end, a uniform image, equal to the average of the initial gray values, is obtained. Although this property is good for local reducing noise (averaging is optimal for
additive noise), this filtering operation also destroys the image content, that is, the boundaries of the objects and the subregions present in the image (the \textit{edges}). This means that the Gaussian smoothing  does not
only smooth noise, but also blurs important features and it
makes them harder to identify. Furthermore, linear diffusion filtering dislocates edges when moving from finer to coarser
scales (see e.g. \cite{Wit83}). So structures which are identified at a coarse
scale do not give the right location and have to be traced back to the original
image.  Moreover, some smoothing properties of Gaussian scale-space do not carry over from
the one-dimensional case to higher dimensions: it is generally not true that the
number of local extrema, which are related to edges, is non-increasing.
As suggested by Hummel \cite{Hum86} the linear diffusion  is not the only PDE that can be used to enhance an image and
that, in order to keep the scale-space property, we need only to make sure
that the corresponding flow  holds the maximum principle.  Many approaches have been taken in the literature
to implement this idea replacing the linear equation with a nonlinear PDE
that does not diffuse the image in a uniform way: these flows are normally
denoted as anisotropic diffusion. In particular, the diffusion coefficient is locally adapted, becoming negligible as
object boundaries are approached. Noise is efficiently removed and object contours are strongly enhanced \cite{Wei98}.
There is a vast literature concerning nonlinear anisotropic diffusions with application to image
processing which date back to the seminal paper by Perona and Malik, who,
in \cite{PM90}, consider a discrete version of the following equation
\[\left\{
\begin{array}{lll}
\frac{\partial u}{\partial t} - \nabla\cdot \left(g(|\nabla u|)\nabla u\right) = 0,\,&\text{in }\, &\Omega_T=(0,T)\times \Omega , \\
\\
u(x,0)=u_0(x) &\, \text{on} & \Omega \\
\\
\frac{\partial u}{\partial \vec{n}} (x,t) =0, \, &\text{on} & \Gamma\times (0,T),
\end{array}\right.
\]
where $ \Gamma = \partial\Omega $, the imaage domain $\Omega \subset \mathbb{R}^2$ is an open regular set (typically a rectangle), $\vec{n}$ denotes the unit outer normal to its boundary $\Gamma$,
and $u(x,t)$ denotes the (scalar) image analysed at time (scale) $t$ and point $x$.  The initial condition  $u_0(x)$,
$u_0$ is, as in the linear case, the original image. In order to reduce smoothing at edges, the diffusivity $g$ is chosen as
a decreasing function of the ``edge detector'' $|\nabla u|$ (for a vector $V=(V_1,V_2)\in \mathbb{R}^2$, $|V|^2=V_1^2+V_2^2$).
A typical choice is,
\[
g(s)=\frac{1}{1+\left(s/\lambda\right)^2},\,s\geq 0,\,\lambda >0.
\]
Catt\'{e}, Lions, Morel and Coll \cite{CLMC1992} showed that the continuous Perona-Malik model is ill posed, then
very close pictures can produce divergent solutions and therefore very different
edges. This is caused by the fact that the diffusivity $g$  leads to flux
$s \cdot g(s)$ decreasing for some $s$ and the scheme may work locally like the inverse
heat equation which is known to be ill posed. This possible misbehaviour surely represents a
severe drawback of the Perona-Malik model when applied to data effected by noise.
However, discrete implementations work as a regularization factor by introducing
implicit diffusion into the model, and the filter is usually observed to be stable
(with staircasing effect as the only observable instability). Then, in the continuous settings,
a new model has been proposed \cite{CLMC1992} with the only modification of replacing the gradient $\nabla u$ in the diffusivity
by its  spatial regularizations $(G_\sigma * \nabla u)$, which are obtained by smoothing the argument by a convolution with a $C^\infty$ kernel $G_\sigma$. Typically $G_\sigma$ is a Gaussian function and $\sigma$ determines the scale beyond which regularization
occurs. The equation  will now diffuse if and only if the gradient is estimated to be small. We point out that the spatial regularization lead to processes where the solution converges to a constant steady state. Then, in order to get nontrivial results, we have  to
specify a stopping time $T$. Sometimes it is attempted to circumvent this task by adding an additional reaction term
which keeps the steady state solution close to the original image $u_0$, for example
\[
\frac{\partial u}{\partial t} - \nabla\cdot \left(g(|\nabla G_\sigma * u|)\nabla u\right) = f(u_0-u),
\]
where $f$ is a Lipshitz continuous, non dereasing funtion, $f (0) = 0$.  During the last
years, many other nonlinear parabolic equations have been proposed as an  image analysis
model. The common theme in this proliferation of models is
the following, one attempts to fix one intrinsic diffusion direction and tunes the diffusion
using the size of the gradient or the value of an estimate of the gradient. A few of the proposed models are
even systems of PDEs, for example  there exist reaction diffusion systems which
have been applied to image restoration and which are  connected to Perona-Malik idea or  based on Turing’s pattern formation model \cite{Wei98}.

In this paper we proposed a new anisotropic diffusion equation introducing a nonlocal diffusive coefficient that takes into account of the ``monotonicity'' and the oscillating pattern of the image. In other words, a high modulus of the gradient may lead to a small diffusion if the function is, for instance, locally monotone.
\subsection{A motivating $1D$ model}
\label{sec:1Dmodel}
At present, the best view of the activity of a neural circuit is provided by multiple-electrode
extracellular recording technologies, which allow us to simultaneously measure spike trains
from up to a few hundred neurons in one or more brain areas during each trial. While the
resulting data provide an extensive picture of neural spiking, their use in characterizing the
fine timescale dynamics of a neural circuit is complicated by at least two factors. First,
extracellularly captured action potentials provide only an occasional view of the process
from which they are generated, forcing us to interpolate the evolution of the circuit between
the spikes. Second, the circuit activity may evolve quite differently on different trials that
are otherwise experimentally identical. Experimental measurements are noisy. For neural recordings, the noise may arise from a
multitude of sources, both intrinsic and extrinsic to the nervous system. Operationally,
supposing that recorded data are composed of two parts, signal of interest and other processes
unrelated to the experimental conditions. it is a challenge to preserve the essential signal features, such as suitable structures related to
the neuronal activity, during the smoothing process.
\begin{figure}[tbhp]
\begin{center}
\includegraphics[width=0.90\textwidth]{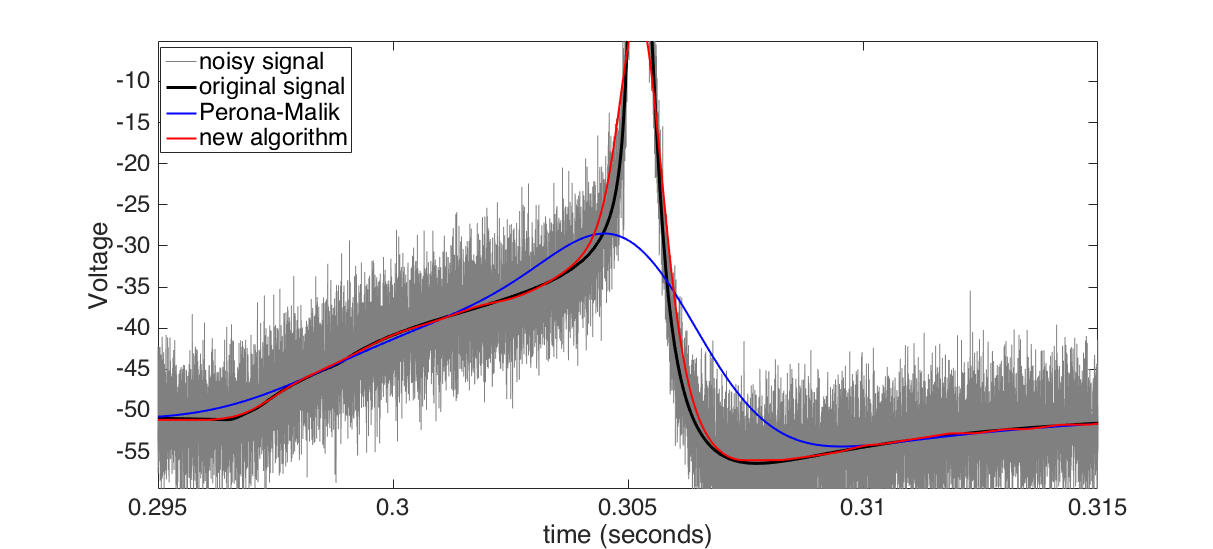}
\end{center}
\caption{$5 s$ simulated membrane potential signals of a network of 20 neurons randomly connected (firing rate $= 5 Hz$). In black, original signal of a neuron.
White gaussian noise was added to the signals (gray).
Blue: signal after denoising (Perona-Malik). Red: signal after denoising (new algorithm).
Note that the event occurred between $t=0.296$
and $t=0.297$ is completely removed after Perona-Malik smoothing}
\label{Fig:1D_smoothing}
\end{figure}
In Figure~\ref{Fig:1D_smoothing}  we show an example of a noised signal of a neuron, where a white gaussian noise
has been superposed to the original signal. The method proposed in this paper is compared with the classical Perona-Malik
algorithm. We point out that the diffusivity in the Perona-Malik model, or similar approaches, depends locally on the modulus
of the gradient of the function. Instead, we introduce a nonlocal diffusive coefficient that takes into account of the ``monotonicity'' of the
signal. In other words, a high modulus of the gradient may lead to a small diffusion if the function is also locally monotone.
Motivated by this fact, we have developed the new approach presented in this paper.
\begin{figure}[tbhp]
\begin{center}
\includegraphics[width=0.95\textwidth]{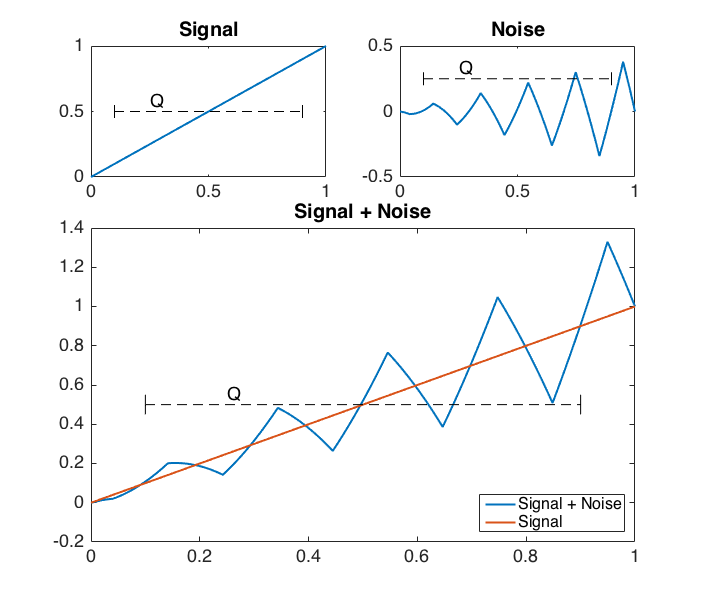}
\end{center}
\caption{Example of a nonlocal signal.
The signal and the noise have the same total variation and the same modulus
of the gradient. The global increment observed on $Q$ is instead very different}
\label{Fig:1D_sign_plus_noise}
\end{figure}
More precisely, the diffusion coefficient in a point $x$ is based on the behavior of the function $f$ in a interval $x+Q=(x-q,x+q)$,
where $Q= (-q,+q)$, see Figure~\ref{Fig:1D_sign_plus_noise}. Analytically, we compute the ratio between the variation $|f(x+q)-f(x-q)|$ and the total variation $\int_{Q} |\nabla f(s+x)| ds $ of the function in $x+Q$. A ratio close to $1$ will imply a tiny noise in the signal, while a ratio close to $0$ is related to a highly noised signal. As shown in Figure~\ref{Fig:1D_sign_plus_noise}, a pure signal and a noised one may have the same total variation and the same modulus of gradient. Therefore, Perona-Malik like methods (and total variation based methods) treat the signals in the same way.\\
More precisely, for the one dimensional spatial case, let  $u:[a,b] \rightarrow\mathbb{R}$ a real function defined on a bounded interval $[a,b]$, and a subinterval
$[c,d] \subset [a,b ]$. We define the \textit{local variation} $LV_{[c,d]}(u)$  of $u$ on the interval $[c,d]$ the value
\[
LV_{[c,d]}(u) = \lvert u(d) - u(c)  \rvert .
\]
We also define the \textit{total local variation} $TV_{[c,d]}(u)$ of $u$ on the interval $[c,d]$  as follows
\[
TV_{[c,d]}(u) = \sup_{\mathcal{P}} \sum_{i=0}^{n_{P}-1} \lvert u ( x_{i+1}) - u( x_{i})  \rvert
\]
where
$\mathcal{P} = \left\lbrace P = \left\lbrace x_{0}, \dots , x_{n_{P}} \right\rbrace | P \text{  is a partition of  } [ a,b ] \right\rbrace $
is the set of all possible finite  partition of the interval $[c,d]$.
It is easly to prove that if the function $u$ is a monotone function on the interval $[c,d]$, then  $LV_{[c,d]}(u)= TV_{[c,d]}(u)$. While, if the function $u$ is not monotone,  $LV_{[c,d]}(u) < TV_{[c,d]}(u)$.
For the $1D$ signal, as the membrane potential of a neuron,  where the independent variable has the dimension of a time,
it is convenient to select instead of a symmetric window $Q$ an  an asymmetric interval of a given length $\delta$.
Let $\varepsilon \in \mathbb{R}^{+}$  ``small'' number, and let $\delta \in \mathbb{R}^{+}$ a positive number,
we define the ratio,
\begin{equation} \label{eq: LVsuTV}
R_{\delta ,u} = \dfrac{LV_{[ x,x+\delta ]}(u)}{\varepsilon + TV_{[x,x+\delta ]}(u)}
\end{equation}
If the parameter $\delta$  is chosen appropriately we can distinguish between oscillations caused by noise and by
electrophysiological stimuli (in the following EPSP) contained in a range of amplitude $\delta$.
In the case of the membrane potential of a neuron the oscillations due to the noise and to EPSP occur on different time scales: it is possible to choose a value $\delta$ such that in a range of amplitude $\delta$ there is at least a full oscillation due to noise, but not to a complete EPSP.
Then, there is an oscillation, the signal is not  monotone and it is expected that the ratio $R_{\delta ,u}$ is much less than one because
$LV_{[ x,x+\delta ]}(u)<< TV_{[x,x+\delta ]}(u)$. While, if in the same time interval there is an EPSP, the ratio $R_{\delta ,u}$
becomes close to one.

As in the Perona-Malik model, we adapt the diffusive coefficient by using the above ratio $R_{\delta ,u}$. For small values of the latter we have
to reduce the noise, while for values close to $1$, the upper bound of $R_{\delta ,u}$, we have to preserve the signal variation (as the edges in the image). The resulting non-local equation is the following,
\begin{equation}\label{eq_nonlocal_1D}
\frac{\partial u}{\partial t} - \frac{\partial}{\partial x} \left( g\left( \frac{LV_{[x,x+\delta ]}(u)}{\varepsilon + TV_{[x, x+\delta ]}(u)} \right)  \frac{\partial u}{\partial x} \right)=0,
\end{equation}
where the function $g$ has the same properties as in the Perona-Malik model and $\delta >0$. If $u$ is a  differentiable function
and $u^\prime$ is integrable, the  total variation cab be written as,
\[
 TV_{[x, x+\delta ]}(u) = \int_x^{x+\delta} |u^\prime (s)| ds,
\]
and the non linear diffusion equation (\ref{eq_nonlocal_1D}) can be state as
\begin{equation*}
\frac{\partial u}{\partial t} - \frac{\partial}{\partial x} \left( g\left( \frac{\left|\int_{x}^{x+\delta} u^\prime (s) ds \right|}{\varepsilon + \int_{x}^{x+\delta} \left| u^\prime (s) \right| ds} \right)  \frac{\partial u}{\partial x} \right) =0.
\end{equation*}
for a function $u(x,t)$, $x\in (a,b)$, $t>0$. As initial condition we take the original signal $u_0$ but with some regularization obtain with a standard smoothing filter, e.g. a Gaussian filter, and we assume homogeneous Neumann condition at the boundary, that is $\partial u/\partial x=0$ for $x=a,b$ and $t>0$.
\subsection{The multidimensional case}
\label{subsec:nD}
In order to apply the new model to a multidimensional signal, in particular in the two-dimensional case (a gray level digital image), we have to generalize the ratio $R_{\delta ,u}$, see (\ref{eq: LVsuTV}). Let $A \subseteq \mathbb{R}^d$ and $u:\Omega \rightarrow \mathbb{R}$ an integrable function smooth function, the total variation $TV(u)$ (or $BV$ seminorm), can be computed as \cite{Zie1989}
\[
TV(u) = \int_A \left|  \nabla u   \right| dx
\]
where $\nabla u$ is the gradient of $u$. Here, we consider the anisotropic total variation,
\[
TV_a(u) = \int_A \left|  \nabla u   \right|_1 dx
\]
considering the $l^1$ norm, $|v|_1=|v_1|+|v_2|+\ldots +|v_d|$, instead of the Eucledian norm.  The usual total variation $TV (u)$ is invariant to rotation of the domain, but anisotropic $TV_a$ is not. However, the latter  allows for other approaches that do not apply with the usual $TV$, for example the  graph-cuts algorithm \cite{CD2009}. Moreover the $TV_a(u)$ has the advantage of making the total variation satisfies the coarea formula \cite{EG1992},
which allows us to interpret $TV_a(u)$ as the cumulated length of the level lines of the function $u$. \\
For the local variation term, the numerator of the ratio $R_{\delta ,u}$, we have to compute the variation of the function $u$ in a region $A$
by taking into account the flux of $u$ at the boundary $\partial A$ of the same set $A$. Following the definition of the $BV-$seminorm \cite{Zie1989}, and the  choice we propose the definition,
\[
LV_A(u)= \sup\{ \int_{A} \nabla u(x) \nabla h(x)\, dx, \ | \nabla h |_1 \leq 1 , h \text{ harmonic on }A\} .
\]
In the above definition, due to the properties of the test function $h$, we have
\[
\int_{A} \nabla u \nabla h\, dx = \int_{\partial A}  u \nabla h \cdot \vec{n}_A\,ds - \int_{A}  u \,div(\nabla h) \, dx
\]
where $div$ is the divergence operator, $\vec{n}_A$ denotes the unit outer normal to $\partial A$, and the last integral is equal to zero because $h$ is harmonic. Then
\[
\int_{A} \nabla u \nabla h\, dx = \int_{\partial A}  u \nabla h \cdot \vec{n}_A\,ds,
\]
and the supremum for the $LV_A(u)$ is taken considering all the possible orientations of the vector $\nabla h$ with respect to $\vec{n}_A$.
Returning to the one-dimensional case, for $A=[c,d]$, we obtain,
\[
LV_A(u) = \sup \{ (u(c)-u(d)),\,(-u(c)+u(d))\} = | u(d)-u(c)|.
\]

The remainder of this paper is organized as follows. In Section \ref{sec:analysis}, we provide the mathematical analysis of the new non-linear and non-local diffusion equation in the two dimensional spatial case. We show in particular the existence of a solution for the model by using a suitable semidiscrete scheme under reasonable hypotheses for applications in image processing.  In Section \ref{sec:NR}, we build an explicit, in time,  numerical scheme for the new model coupling a finite element method based on bilinear element $\mathcal{Q}_1$, a finite difference approximation for the numerical gradients, and a decomposition with respect to a suitable set of eigenfunctions. In Section \ref{sec:NR} we also show some numerical
experiments.
\subsection*{Notation}
In the following, $\Omega\subset \mathbb{R}^2$  denotes a open bounded domain with Lipschitz continuous boundary  $\Gamma = \partial \Omega$, and
$\Omega_T=\Omega \times (0,T)$, with $T>0$.
We denote by $H^k(\Omega )$, $k$ is a positive integer, the Sobolev space  of all function $u$ defined in $\Omega$ such that $u$ and its distributional derivatives of order $k$ all belong to $L^2(\Omega )$. Let $D^s$ the distributional derivatives,  $H^k$ is a Hilbert space for the norm,
\[
\| u \|_k = \| u \|_{H^k} = \left(  \sum_{|s|\leq k} \int_{\Omega} \left| D^s u(x) \right|^2 dx  \right)^{1/2} , \qquad
\| u \|_0 = \| u \|_{L^2} .
\]
Let $V=H^1$, $V^*$ stands for
its dual space. 
We denote by $L^p(0, T; H^k(\Omega ))$ the set of all functions $u$, such that, for almost every $t$ in $(0, T)$, $u(t)$ belong to $H^k(\Omega)$, $L^p(0, T; H^k(\Omega ))$ is a normed space for the norm
\[
\|u \|_{L^p(0, T; H^k(\Omega ))} = \left( \int_{0}^{T}\| u \|^p_k  dt \right)^{1/p}
\]
$p \geq 1$ and $k$ a positive integer.
We denote by $(\cdot , \cdot )$, the scalar product in $L^2(\Omega )$.
\section{Analysis of the new nonlocal and nonlinear equation, $2D$ case}
\label{sec:analysis}

In this section we will consider the two-dimensional spatial case and we will prove  the existence of a variational  solution of the corresponding non-local diffusion equation.  From the discussion in the subsections (\ref{sec:1Dmodel})-(\ref{subsec:nD}), given $U \in L^2 (0, T ; V)$ and $Q=  (-q_1,+q_1) \times (-q_2,+q_2) $ (the local window), we can define the ratio coefficient $R$ as the function
\begin{small}
\[
R_{Q,U}(x,t) =
\begin{cases}
\frac{ \sup\{ \int_{Q} \nabla U (x+y,t) \nabla  h(y)\, dy, \ | \nabla h |_1 \leq 1 , h \text{ harmonic on }Q\} }{ \int_{x+Q}
|\nabla U (y,t)|_1 dy}, & \text{if }\int_{x+Q}
|\nabla U (y,t)|_1 dy >0;
\\
0, & \text{otherwise};
\end{cases}
\]
\end{small}
where $|\cdot|_1$ is the $l^1$ norm in $\mathbb{R}^2$. It is easy to verify that the function $R_{Q,U}(x,t)$ is measurable, and $0\leq R_{Q,U} \leq 1$.
Moreover, note that $\int_{x+Q} |\nabla U (y,t)|_1 dy$ is continuous in $x$ since $U \in L^2 (0, T ; V)$.

Let $g:[0,+\infty) \longrightarrow \mathbb{R}$ be
a Lipschitz continuous  nonincreasing function such that $g(0)=1$, $g(s)>0,\,\forall s\geq 0$, $g(1)=\epsilon >0$.
It follows that
$1 \geq g(R_{Q,U}(x,t)) \geq \epsilon$.

Let $Q$ be the window that is used in the definition of the diffusive coefficient
$R_{Q,u}$. We assume that
\begin{equation}\label{ass:Q-O}\tag{Assumption~1}
\inf_{x\in{{\Omega}}} \frac{|{{\Omega}} \cap \{x+Q/3\}|}{|\{x+Q/3\}|} = q_{{{\Omega}}} >0,
\end{equation}
where if $A$ is a measurable set, let $|A|$ be the Lebesque measure of $A$.

The smoothing process of the image $u_I$ is obtained by the  solution $u(x,t)$ of the following non-linear, non-local diffusion
equation,
\begin{equation}\label{eq.nonlinear}
\begin{aligned}
&
\frac{\partial u}{\partial t }-\mathrm{div} \left( g(R_{Q,u}(x,t)) \nabla u\right) =0, &&\text{in } \Omega_T;
\\
&
\frac{\partial u}{\partial \vec{n}} =0, && \text{on } \Gamma \times (0,T);
\\
& u(x,0)=u_0(x)\in V;
\end{aligned}
\end{equation}
with  homogeneous Neumann boundary conditions for the normal derivative $\frac{\partial u}{\partial \vec{n}}$, and initial data $u_0\in V$ which is a smoother version of the original image $u_I$.
\begin{remark}
We point out that the initial data  is more regular with respect to classical parabolic theory but we need to ensure the well-posedness
of the diffusion coefficients  $ R_{Q,u}$. In the numerical approximation of the equation (\ref{eq.nonlinear}) we obtain a suitable
initial data from the original signal $u_I$  by using a convolutional operator with a Gaussian filter.
\end{remark}
\subsection{Rothe method and a priori estimates}

In order to prove the existence of a solution $u\in L^2 (0, T ; V)\bigcap C^0(0,T ; L^2)$ we consider
the so called  Rothe-type approximation \cite{Kacur85}
of (\ref{eq.nonlinear})  which consists  in using time discretization to approximate the evolution
problem by a sequence of  an elliptic one.
To show the convergence of such a process, a common approach is to follow the following steps:
\begin{enumerate}
  \item  for each approximate problem, prove the existence of a solution, and
derive a-priori estimates satisfied by any solution;
  \item then use compactness arguments to show (up to the extraction of a
subsequence) the existence of a limit;
  \item Finally, prove that the previous limit satisfies the original problem.
\end{enumerate}
Let $0 = t_0 < t_1 < \ldots < t_N = T$ denote the time discretization with $t_{i+1}=t_i+\tau$, where $\tau$ is the time step.
Let $u_i$ be the solution of linear equation,
\begin{equation}\label{eq.linear}
\frac{u_i-u_{i-1}}{\tau}-\mathrm{div} \left( g(R_{Q,u_{i-1}}(x,t)) \nabla u_i \right) =0,
\end{equation}
with $u_0=u(x,0)$, and homogeneous Neumann boundary conditions.  Let $\delta u_i=(u_i-u_{i-1})/\tau$ the backward difference at time $t_i$,
we understand the solution of (\ref{eq.linear}) in the variational sense, i.e., we look for $u_i\in V$, for $i=1,\ldots ,N$ satisfies the identity
\begin{equation}\label{eq.linear.variational}
\left( \delta u_i, v \right) + \left( g(R_{Q,u_{i-1}}) \nabla u_i, \nabla v \right) =0,\,\, \forall v\in V,
\end{equation}
where $u_0\in V$ is given. By introducing the bilinear form $a_{\tau,w}$, on $V\times V$,
\begin{equation*}
a_{\tau,w}=(u,v)+\tau a_w(u,v),\qquad  a_w(u,v)= (g(R_{Q,w}) \nabla u, \nabla v),
\end{equation*}
for a given $w\in V$,  we can rewrite the previous identy as,
\begin{equation}\label{variational-2}
a_{\tau , u_{i-1}} (u_i,v)=(u_{i-1},v),\, \forall v\in V.
\end{equation}
The term $a_w(u,v)$ in (\refname{atau-form}) is weakly coercive, i.e., there exist two constants $c_1 > 0$ and $c_2 > 0$ such that
\begin{equation}\label{coerciva-aw}
a_w(u,u) + c_2 \|u\|_0^2 \geq c_1 \| u \|^2_1,\, \forall u \in V.
\end{equation}
Furthermore, the form $a_{\tau,w}$ is continuous and it verifies
\begin{equation*}
a_{\tau,w}(u,u)\geq \tau c_1 \| u \|^2_1 + (1-\tau c_2) \|u\|_0^2,
\end{equation*}
then it is $V-$elliptic if $\tau c_2 \leq 1$. Under this coercivity condition, the existence and the uniqueness of
$u_i\in V$, $i=1,\ldots , N$ from (\ref{eq.linear.variational}) is guaranteed by the Lax-Milgram Theorem \cite{RT2004}.
Now, we introduce the so-called Rothe function
\begin{equation*}
u^{(N)}(t) =u_{i-1} +(t-t_{i-1}) \delta u_i,\,\,\text{for }t_{i-1} \leq t \leq t_i,\, i=1,\ldots , N
\end{equation*}
which we consider as a linear piecewise approximation of the problem (\ref{eq.nonlinear}). Together with $u^{(N)}$ we consider
the step function
\begin{equation*}
\bar{u}^{(N)}(t) =u_{i},\,\,\text{for }t_{i-1} \leq t \leq t_i,\, i=1,\ldots , N
\end{equation*}
with $\bar{u}^{(N)}(0) =u_{0}$. In the following $C$ denotes the generic positive constant.
\begin{lemma}\label{lemma-stima-ui}
Let $u_i$, $i=1, \ldots , N$,   be the solution of problem   (\ref{variational-2}), then
\begin{equation}\label{stima-ui}
\max_{1\leq i \leq N} \|u_i\|_0 \leq C,
\end{equation}
hold uniformly for $N$, and $\bar{u}^{(N)}(t),~u^{(N)} (t) \in L^\infty (0,T; L^2)$.
\end{lemma}
\begin{proof}
First we test (\ref{eq.linear.variational}) at time $t_{k+1}$ by $v=\tau u_{k+1}$ and sum over $k=0,\ldots , p\geq 1$, we obtain (let us define $a_{u_k}(u,v)=a_k(u,v)$),
\begin{equation*}
\sum_{k=0}^{p} (u_{k+1}-u_k, u_{k+1}) + \tau \sum_{k=0}^{p} a_k(u_{k+1},u_{k+1}) = 0,\qquad
p=1,\ldots (n-1).
\end{equation*}
By using the identity $2(u-v,u)=(u-v,u-v)+(u,u)-(v,v)$, we have
\begin{equation*}
\sum_{k=0}^{p} \|u_{k+1}-u_k\|_0^2 + \|u_{p+1}\|_0^2 - \|u_{0}\|_0^2
+ 2\tau \sum_{k=0}^{p} a_k(u_{k+1},u_{k+1}) = 0,
\end{equation*}
Then, from (\ref{coerciva-aw}),
\begin{equation*}
\sum_{k=0}^{p} \|u_{k+1}-u_k\|_0^2 + \|u_{p+1}\|_0^2 - \|u_{0}\|_0^2
+ 2\tau c_1 \sum_{k=0}^{p} \| u_{k+1}\|_1^2
-2 \tau c_2  \sum_{k=0}^{p}\| u_{k+1}\|_0^2 \leq  0,
\end{equation*}
which can be rewritten as follows
\begin{equation*}
\sum_{k=0}^{p} \|u_{k+1}-u_k\|_0^2
+ \|u_{p+1}\|_0^2  + 2\tau c_1 \sum_{k=0}^{p} \| u_{k+1}\|_1^2 \leq
2 \tau c_2  \sum_{k=0}^{p}\| u_{k+1}\|_0^2 + \|u_{0}\|_0^2 .
\end{equation*}
Let us define
\begin{equation*}
s_p = \sum_{k=0}^{p-1} \|u_{k+1}-u_k\|_0^2 +
\frac{1}{2} \|u_{p+1}\|_0^2  + 2\tau c_1 \sum_{k=0}^{p-1} \| u_{k+1}\|_1^2,
\end{equation*}
It is easy to verify that
\begin{equation*}
s_{p+1} \leq C_0 + 2\tau c_1 \sum_{k=0}^{p} \| u_{k}\|_1^2,
\end{equation*}
where the constant $C_0$ depends only on the initial datum $u_0$. By the  inequality
$2s_p \geq \|u_p\|_0^2$, we obtain
\begin{equation*}
s_{p+1} \leq C_0 + 4\tau c_1 \sum_{k=0}^{p} s_k.
\end{equation*}
Applying the discrete Gronwall lemma we have the following inequalities,
\begin{equation*}
s_{p} \leq C_0 \left( 1 + 4\tau c_1 \right)^{p-1}\leq C_0 \left( 1 + 4\tau c_1 \right)^{N}.
\end{equation*}
The function $\tau \rightarrow \left( 1+ 4\tau c_1 \right)^{T/\tau}$ is bounded for $\tau \in (0,+\infty )$, then, for a suitable constant $\bar{C}$,
\begin{equation}\label{eq.Gro.4}
\max_p \frac{\|u_p\|_0^2}{2} \leq \max_p s_{p} \leq C_0 \bar{C} =  C,
\end{equation}
which leads to the a-priori estimates in \eqref{stima-ui}.
\end{proof}
\begin{lemma}\label{lem.second.es}
The estimates
\begin{equation}\label{eq.estimates.2}
\tau \sum_{k=1}^{N} \| \nabla u_k \|_0^2 \leq C,\qquad
\sum_{k=1}^{N}\|u_k-u_{k-1}\|_0^2 \leq C
\end{equation}
hold uniformly  with respect to $N$.
\end{lemma}
\begin{proof}
The estimates followed by the upper bound obtained  in the Lemma \ref{lemma-stima-ui}, see (\ref{eq.Gro.4}),
\begin{equation*}
\sum_{k=0}^{p} \|u_{k+1}-u_k\|_0^2 +
\frac{1}{2} \|u_{p+1}\|_0^2  + 2\tau c_1 \sum_{k=0}^{p} \| u_{k+1}\|_1^2 \leq C,
\end{equation*}
and from the properties of the bilinear form $a_k(u,v)$.
\end{proof}
\subsection{Compactness and passage to the limit}

We remember that $\Omega_T=\Omega \times (0,T)$, the estimates (\ref{stima-ui},\ref{eq.estimates.2}) will lead the equicontinuity of  the Rothe approximation
$u^{(N)}$,  together the step function $\overline{u}{(N)}$.
\begin{lemma}\label{lem:prefinale}
 There exists $u\in L^2(0,T; V)$ with $\partial u /\partial t \in L^2(0,T; V^*)$ such that (in the sense of subsequences)
\[
 \begin{aligned}
&u^{(N)} \rightarrow u,\quad \overline{u}^{(N)}\rightarrow u
&& \text{in } L^2(\Omega_T); \\
&
\overline{u}^{(N)}\rightharpoonup u && \text{in } L^2(0,T;V); \\
&
\frac{\partial u^{(N)}}{\partial t} \rightharpoonup \frac{\partial u}{\partial t}
&& \text{in } L^2(0,T;V^*).
\end{aligned}
\]
\end{lemma}
\begin{proof}
 Let $s\in (0,T)$, we consider the time translate variation of the approximation $\overline{u}^{(N)}(t)$,
\begin{equation*}
 J_s= \int_0^{T-s} \| \overline{u}^{(N)}(t+s) - \overline{u}^{(N)}(t) \|_0^2 dt.
\end{equation*}
There exists an integer $k$ such that $k\tau \leq s \leq (k+1)\tau$, then, by the definition of the step function  $\overline{u}^{(N)}$,
\[
 J_s= \tau \sum_{l=0}^{N-k} \| u_{l+k} - u_l\|_0^2.
\]
From the Lemma \ref{lemma-stima-ui} and \ref{lem.second.es} it follows that
\[
J_s \leq Ck\tau,
\]
for a suitable constant $C$, and the time translate estimate,
\begin{equation}\label{eq.time.trans.est}
  \int_0^{T-s} \| \overline{u}^{(N)}(t+s) - \overline{u}^{(N)}(t) \|_0^2 dt \leq C (s+\tau ).
\end{equation}
By using again the estimates (\ref{stima-ui}), (\ref{eq.estimates.2}), and the definition of the step approximation
$\overline{u}^{(N)}$, it is easy to show that
\begin{equation}\label{eq.est.H1norm.ubar}
\int_0^T \| \overline{u}^{(N)} (t) \|_1^2 dt \leq C_u.
\end{equation}
Given a vector $\xi \in \mathbb{R}^2$, let
$\Omega_\xi=\{  x\in \Omega \colon x + \xi \in\Omega \}$ and
$\Omega_{\xi ,T} = \Omega_{\xi} \times (0,T)$.
From the
inequality (\ref{eq.est.H1norm.ubar}) we have the following space translate estimate,
\begin{equation}\label{eq.space.trans.est}
\int_{\Omega_{\xi ,T}} \| \overline{u}^{(N)} (x +\xi , t) - \overline{u}^{(N)} (x, t) \|_0^2\, dxdt
\leq C_\xi |{\bf \xi} |^2 , \quad \text{for } |\xi| \text{ sufficiently small}.
\end{equation}
Due to the  time and translate estimates, respectively  (\ref{eq.time.trans.est}) and (\ref{eq.space.trans.est}), the set
$\{  \overline{u}^{(N)} \}_N$ is compact in $L^2(\Omega_T)$ because of Kolmogorov's relative compactness
criterion \cite{Brezis2010,McO1996}.
The, we can conclude $\overline{u}^{(N)} \longrightarrow u$ in $L^2 (\Omega_T)$ (and also
pointwise in $\Omega_T$),and $u\in L^2(0,T;V)$.
From the definition of $\overline{u}^{(N)}$ and ${u}^{(N)}$, and from Lemma \ref{lem.second.es} it follows the estimate
\begin{equation*}
\int_0^T \| \overline{u}^{(N)} (t) - {u}^{(N)}(t)\|_0^2 dt \leq \frac{C_d}{N},
\end{equation*}
which holds uniformly with respect to $N$. Then ${u}^{(N)} \longrightarrow u$ in $L^2 (\Omega_T)$. \\
Observing that $\partial  {u}^{(N)}/\partial t = (u_i - u_{i-1})/\tau $, for $t\in (t_{i-1}, t_i)$, we can compute
\[
\Big\| \frac{\partial  {u}^{(N)}}{\partial t} \Big\|_* = \sup_{v\in V,~\| v\|\leq 1}~\left| ((u_i - u_{i-1})/\tau  ,v)  \right|.
\]
Then, the following estimate holds uniformly for $N$,
\begin{equation*}
\int_0^T \Big\| \frac{\partial  {u}^{(N)}}{\partial t} \Big\|_*^2  \, dt \leq C_*,
\end{equation*}
and we can deduce the weak convergence of the time derivative of the Rothe approximations $ {u}^{(N)}$.
\end{proof}

\begin{lemma}\label{lem:StrongConv}
With the notation of Lemma~\ref{lem:prefinale},
\[
u^{(N)} \rightarrow u,\quad \overline{u}^{(N)}\rightarrow u
\qquad \text{in } L^2(0,T;V);
\]
\end{lemma}
\begin{proof}
Now we shall prove the $L^2(0,T;V)$ convergence of  $\overline{u}^{(N)}$ to $u$ (which  belongs to the space $L^2(0,T;V)$).
So, let us test  (\ref{eq.linear.variational}) by $v=\overline{u}^{(N)} -u$ and integrate it over the time interval $(0,T)$ by
using the partition from the subinterval $(t_{i-1},t_i)$,
\begin{equation}\label{eq.sum.1}
\sum_i^N  \int_{t_{i-1}}^{t_i} \left[ \left( (u_i-u_{i-1})/\tau , u_i -u  \right) + \left( g(R_{Q,u_{i-1}}) \nabla u_i, \nabla u_i -\nabla u
\right)\right] dt =0.
\end{equation}
We recall that $ 1 \geq g(R_{Q,u_{i-1}}) \geq g(1)=\epsilon$, then we obtain
\[
\epsilon \int_0^T \| \nabla  \overline{u}^{(N)} - \nabla u \|_0^2 dt \leq \sum_i^N  \int_{t_{i-1}}^{t_i} \left( g(R_{Q,u_{i-1}})
\nabla u_i- \nabla u , \nabla u_i -\nabla u  \right) dt.
\]
From (\ref{eq.sum.1}), and the above inequality, we have
\begin{multline*}
\epsilon \int_0^T \| \nabla  \overline{u}^{(N)} - \nabla u \|_0^2\, dt + \sum_i^N
\int_{t_{i-1}}^{t_i} \left( g(R_{Q,u_{i-1}})
\nabla u_i, \nabla u_i -\nabla u  \right) dt \\
\leq
\sum_i^N  \int_{t_{i-1}}^{t_i} \left( (u_i-u_{i-1})/\tau , u -u_i  \right) dt.
\end{multline*}
Now, from the Lemma \ref{lemma-stima-ui} and the Lemma \ref{lem.second.es}, and the convergence $\overline{u}^{(N)}\longrightarrow u$ in $L^2(\Omega_T)$,
and the weak convergence $\partial \overline{u}^{(N)} /\partial t \rightharpoonup \partial u /\partial t$ in the space
$L^2(0,T;V^*)$, we deduce that exists a vanishing sequence $\{ C_N \},~C_N\in\mathbb{R},~\lim_{N\rightarrow\infty}C_N =0$ such that
\[
\epsilon \int_0^T | \nabla  \overline{u}^{(N)} - \nabla u |^2 dt \leq C_N,
\]
which implies $ \overline{u}^{(N)} \rightarrow u$ in $L^2(0,T;V)$. To prove the convergence ${u}^{(N)} \rightarrow u$ in $L^2(0,T;V)$ it is possible
to consider a time average approximation starting from the values $u_i$ and proceeding as in \cite{KM1995}.
\end{proof}
\subsection{Existence of a variational solution}

In order to prove that the limit $u$ is a variational solution of (\ref{eq.nonlinear}) we have to consider
the property of the stability of the kernel $g(R_{Q,u})$ with respect a variation in the space $V$.
First, we introduce some results from the measure theory.

 \begin{lemma}[Vitali covering]
Let $\cup_{i=1}^n x_i+Q/3$ be a finite cover of a set $\tilde{{{\Omega}}}\subseteq {{{\Omega}}}$.
Then there exists
a finite sub-cover $\cup_{j=1}^m \{x_j+Q\}$ such that $\{x_j+Q/3,j=1,\ldots,m\}$ are disjoint.
\end{lemma}

\begin{corollary}\label{cor:uniform_N0}
Let $\tilde{{{\Omega}}}\subseteq {{{\Omega}}}$. Then there exists $x_{0}, \ldots , x_{N_0}\in \tilde{{{\Omega}}}$ such that
$\tilde{{{\Omega}}}\subseteq \cup_{j=1}^{N_0} \{x_j+Q\} $ and
$N_0 \leq \frac{3^2 |{{\Omega}}|}{|Q| q_{{\Omega}}} $.
\end{corollary}
\begin{proof}
Denote by $\hat{{{\Omega}}}$ the closure of $\tilde{{{\Omega}}}$ in ${{\Omega}}$. Take the open cover of $\hat{{{\Omega}}}$ made by
$\mathcal{C}=\{ x + Q/3, x\in \tilde{{{\Omega}}}\}$.
Since ${\overline{{\Omega}}}$ is compact, then $\hat{{{\Omega}}}$ is compact and hence there exists a finite cover of $\hat{{{\Omega}}}$ made
by $\{\cup_{i=1}^n \{x_i+Q/3\}\}$. By the Vitali covering Lemma, there exists
a finite sub-cover $\cup_{j=1}^{N_0} \{x_j+Q\}$ such that $\{x_j+Q/3,j=1,\ldots,N_0\}$ are disjoint. Moreover,
\[
|{{\Omega}} | \geq |{{\Omega}} \cap (\cup_{j=1}^{N_0} \{x_j+Q/3\})| = \sum_{j=1}^{N_0} |{{\Omega}} \cap\{x_j+Q/3\})| \geq
q_{{\Omega}} N_0 \frac{|Q|}{3^2},
\]
the last inequality being a consequence of \eqref{ass:Q-O}, and
since $|\{x_j+Q/3\}|= |Q/3|= |Q| / 3^2$.
\end{proof}

The following result shows the stability of the kernel $g(R_{Q,u})$ when
the limiting function is not locally constant.
\begin{lemma}\label{lem:continuityR_Q}
The function $R_{Q,U}(x,t) $
is continuous at $U\in V$ on the set
\[
\Big\{(x,t)\colon \int_{x+Q}  |\nabla U (y,t)|_1 dy > 0 \Big\} .
\]
\end{lemma}
\begin{proof}
Denote by
\begin{align*}
N_U(x,t) &= \sup\Big\{ \int_{Q} \nabla U (x+y,t) \nabla h(y)\, dy, \ | \nabla h |_1 \leq 1 , h \text{ harmonic in } Q\Big\},
\\
D_U(x,t) &=  \int_{x+Q}
|\nabla U (y,t)|_1 dy,
\end{align*}
and by
\[
|\nabla u|_1 = \Big| \frac{\partial u}{ \partial x_1} \Big| + \Big| \frac{\partial u}{ \partial x_2} \Big|,
\qquad
|\nabla u|_2 = \sqrt{\Big( \frac{\partial u}{ \partial x_1} \Big)^2+
\Big( \frac{\partial u}{ \partial x_2} \Big)^2},
\]
the seminorm $|\cdot |_1$ and, respectively,  $|\cdot |_2$, then
\(
\frac{|\nabla u|_1}{\sqrt{2}} \leq |\nabla u|_2 \leq |\nabla u|_1 .
\)
Let $U_n \to  U$ in $L^2(0,T;V)$. By definition, for a.e.\ $t\in (0,T)$,
\begin{equation*}
\int_{{{\Omega}}} | \nabla U - \nabla U_n |^2_2 \, dy \to 0  ,
\end{equation*}
and hence
\begin{equation}\label{eq:disDu2}
\frac{1}{2\cdot|{{\Omega}}|} \Big(\int_{{{\Omega}}} | \nabla U - \nabla U_n |_1 \, dy \Big)^2 \leq
\int_{{{\Omega}}} \Big( \frac{| \nabla U - \nabla U_n |_1}{\sqrt{2}}\Big)^2 \, dy \leq
\int_{{{\Omega}}} | \nabla U - \nabla U_n |_2^2 \, dy
\to 0  .
\end{equation}
As a direct consequence,
\begin{equation}\label{eq:conv_denomin}
D_{U_n}(x,t) \to D_U(x,t) >0, \qquad \text{for a.e.\ $(x,t)\colon \int_{x+Q}  |\nabla U (y,t)|_1 dy > 0$}.
\end{equation}
For what concerns $N_U$, if $| \nabla h |_1 \leq 1$ then we get $| \nabla h |_\infty \leq 1$.
Then
\begin{multline*}
\Big|
\int_{Q}  \nabla U_1 (x+ y,t) \nabla h(y)\, dy -
\int_{Q}  \nabla U_2 (x+ y,t) \nabla h(y)\, dy
\Big|\\
\begin{aligned}
&
\leq
\int_{Q}  | \nabla U_1 (x+y,t)  -\nabla U_2 (x+y,t)  \,\nabla h(y)|\, dy
\\
&
\leq
\| \nabla U_1 (x+\cdot,t)  -\nabla U_2 (x+\cdot,t)  \|_1 \,  \|\nabla h \|_\infty
\\
&
\leq
\int_{x+ Q}  | \nabla U_1 (y,t)  -\nabla U_2 (y,t)  |_1\, dy
\end{aligned}
\end{multline*}
By \eqref{eq:disDu2},
\(
N_{U_n}(x,t) \to N_U(x,t)
\), which concludes the proof together with \eqref{eq:conv_denomin}.
\end{proof}

\begin{lemma}\label{lem:a_conv}
Let $u_n\to u$ in $V$. For any $w\in C^1({{{\Omega}}} ; \mathbb{R})$ and $t\in (0,T)$,
\[
\lim_{n\to\infty} \Big(
\int_{ {{{\Omega}}}} g(R_{Q,u_n}) \nabla u_n \nabla w \,dx
- \int_{ {{{\Omega}}}} g(R_{Q,u}) \nabla u \nabla w \,dx\Big)= 0.
\]
\end{lemma}
\begin{proof}
Note that
\begin{multline}\label{eq:lemg1}
\Big|\int_{ {{{\Omega}}}} g(R_{Q,u_n}) \nabla u_n \nabla w \,dx
- \int_{ {{{\Omega}}}} g(R_{Q,u}) \nabla u \nabla w \,dx\Big|
\\
\begin{aligned}
& \leq \int_{ {{{\Omega}}}} | g(R_{Q,u_n}) |\, | (\nabla u_n - \nabla u ) \nabla w| \,dx
\\
& \qquad
+
\int_{ {{{\Omega}}}} | g(R_{Q,u_n}) - g(R_{Q,u})|\, | \nabla u \nabla w | \,dx .
\end{aligned}
\end{multline}
For what concerns the RHS of \eqref{eq:lemg1}, the first term vanishes as $n$ goes to infinity since $g$ is bounded.
For the second term of  the RHS of \eqref{eq:lemg1}, we get
\begin{multline}\label{eq:lemg2}
\int_{{{{\Omega}}}} | g(R_{Q,u_n})  - g(R_{Q,u})|\, | \nabla u \nabla w | \,dx
= \int_{\widetilde{{{\Omega}}}^\epsilon} | g(R_{Q,u_n}) - g(R_{Q,u})|\, | \nabla u \nabla w | \,dx
\\
+
\int_{{{{\Omega}}} \setminus \widetilde{{{\Omega}}}^\epsilon} | g(R_{Q,u_n}) - g(R_{Q,u})|\, | \nabla u \nabla w | \,dx
\end{multline}
where,
\[
\widetilde{{{\Omega}}}^\epsilon = \Big\{x \colon \int_{x+Q} |\nabla u |_1 dy  < \frac{\epsilon |Q| q_{{\Omega}} }{g(0) K 3^2 |{{\Omega}}|} \leq
\frac{\epsilon }{g(0) K N_0}
\Big\},
\]
$K$ is such that $\max(|\tfrac{\partial w}{x_1}|,|\tfrac{\partial w}{x_2}|) \leq K$, and
$N_0$ is defined in Corollary~\ref{cor:uniform_N0}.
The first term of the RHS of \eqref{eq:lemg2} is uniformly bounded (in $n$) by $\epsilon $: by Corollary~\ref{cor:uniform_N0}
\[
\begin{aligned}
\int_{\widetilde{{{\Omega}}}^\epsilon} | g(R_{Q,u_n}) - g(R_{Q,u})|\, | \nabla u \nabla w | \,dy
&\leq
g(0) K \int_{\widetilde{{{\Omega}}}^\epsilon }  | \nabla u |_1 \,dy\\
&\leq
g(0) K \int_{ \cup_{j=1}^{N_0} \{x_j+Q\} } | \nabla u |_1 \,dy\\
&\leq
g(0) K \sum_{j=1}^{N_0} \int_{x_j+Q } | \nabla u |_1 \,dy\\
&\leq
g(0) K N_0 \int_{x_j+Q } | \nabla u |_1 \,dy\\
& \leq \epsilon .
\end{aligned}
\]
The second term of the RHS of \eqref{eq:lemg2} vanishes as a consequence of the Dominated Convergence Theorem. In fact,
$| g(R_{Q,u_n}) - g(R_{Q,u}) \nabla u \nabla w| \leq g(0) |\nabla u \nabla w|\in L^1 $;
and $g(R_{Q,u_n}) \to g(R_{Q,u})$ on ${{{\Omega}}} \setminus \widetilde{{{\Omega}}}^\epsilon$  by Lemma~\ref{lem:continuityR_Q}.
\end{proof}

Now we prove that the limit $u$ is a variational solution of (\ref{eq.nonlinear}).

\begin{lemma}[Existence]
For any $u_0 \in V$,
there exists $u\in L^2 (0, T ; V)$ with $\frac{\partial u}{\partial t} \in L^2 (0, T ; V^*)$ such that
\(u(x,0)=u_0(x)\) on $\Omega$,
$\frac{\partial u}{\partial \vec{n}} =0$ on $\Gamma \times (0,T)$ and
\[
\int_\Omega \frac{\partial u}{\partial t}  w \, dx = \int_{ {{{\Omega}}}} \mathrm{div} (g(R_{Q,u}) \nabla u ) w \,dx,  \qquad \forall w \in C^1_0(\Omega).
\]
\end{lemma}
\begin{proof}
By Lemma~\ref{lem:prefinale} and Lemma~\ref{lem:StrongConv}, there exists a
sequence ${u}^{(N)}$ such that
\[
{u}^{(N)} \to u \text{ in } L^2 (0, T ; V),
\qquad
\frac{\partial u^{(N)}}{\partial t} \rightharpoonup \frac{\partial u}{\partial t}
\text{ in } L^2 (0, T ; V^*)
.
\]
Let $\phi \in C^{\infty}_c(0, T )$ be a real-valued test function and
$ w \in C^1_0(\Omega)$.
Taking $v(x,t) = \phi(t) w(x)$ as a test function and integrating the result with respect to $t$,
we find that
\[
\int_0^T \Big( \int_\Omega
\frac{\partial u^{(N)}}{\partial t} v(x,t) \, dx + \int_\Omega
g(R_{Q,u_N}) \nabla u^{(N)} \nabla v(x,t) \, dx \Big) \, dt= 0.
\]
We take the limit of this equation as $N \to\infty$. Since the function $t \to \phi(t)w$ belongs
to $L^2(0, T ; V)$, we have
\begin{equation*}
\int_0^T \int_\Omega \frac{\partial u^{(N)}}{\partial t}  v(x,t) \, dx \, dt
\longrightarrow
\int_0^T \int_\Omega \frac{\partial u}{\partial t}  v(x,t) \, dx \, dt .
\end{equation*}
Moreover, Lemma~\ref{lem:a_conv} shows that
\[
\int_0^T \Big( \int_{ {{{\Omega}}}} g(R_{Q,u_N}) \nabla u^{(N)} \nabla w \,dx
\Big) \phi(t) \,dt
\to \int_0^T \Big( \int_{ {{{\Omega}}}} g(R_{Q,u}) \nabla u \nabla w \,dx \Big) \phi(t) \,dt
\]
It therefore follows that $u$ satisfies
\[
\int_0^T \phi(t) \Big( \int_\Omega \frac{\partial u^{(N)}}{\partial t}  w \, dx +
\int_{ {{{\Omega}}}} g(R_{Q,u}) \nabla u \nabla w \,dx \Big) \, dt =
0, \qquad \forall \phi \in C^{\infty}_c(0, T ),
\]
and hence, for almost every $t\in (0,T)$,
\[
\int_\Omega \frac{\partial u}{\partial t}  w \, dx = \int_{ {{{\Omega}}}} \mathrm{div} (g(R_{Q,u}) \nabla u ) w \,dx,  \qquad \forall w \in C^1_0(\Omega).
\]
\end{proof}
\section{Numerical approximation}
\label{sec:NR}

We introduce here a numerical scheme for the one-dimensional and the two-dimensional case, that is for $1D$ signal and for the gray level images.
We point out that  a digital signal/image is usually defined on a uniform subdivision of an interval or a  rectangular domain. While for the  one-dimensional case we consider a finite differences approach, for the $2D$ spatial problem we have to couple different numerical  approaches in order to
estimates the function $R_{Q,u}$ and the corresponding diffusive coefficient.
\subsection{$1D$ problem}

We approximate the $1D$ non linear diffusion equation (\ref{eq_nonlocal_1D}), for $x\in [0,L]$, $t\in [0,T]$, coupled with homogeneous Neumann conditions for $x=0$, $x=L$, using a finite difference scheme. For convenience, we will use a uniform grid, with grid
spacing $h=L/N$, $N>1$ integer. If we wish to refer to one of the points in the spatial grid, we shall call the points $x_i$, $i=0,\ldots ,N$, where $x_i=ih$,  obtaining  the corresponding partition
$0 = x_{0}<x_1< \ldots < x_{N} = L$. Likewise, we discretize the time domain $[0,T]$ similarly by place a grid on the temporal axis with grid spacing $\tau =T/M$ and time points $t_k=k\tau$, $k=0,\ldots ,M$. We will define $u_{i}^{k}$ to be a function defined at the point
$(ih,k\tau )$, the function $u_{i}^{k}$ will be our approximation to the solution of the
diffusion problem at the same point $(ih,k\tau )$.\\
For simplicity we choose the parameter $\delta$ (the length of the window for the non local term) as  $\delta = lh$, where $l \in \mathbb{N}$.
We approximate the local variation $LV_{\left[ x_{i},x_{i}+\delta\right] }\left( u\right)$ by $LV_{i}= \lvert u_{i+l}-u_{i} \rvert$, and
we define $TV_{i}=\sum_{j=0}^{l-1} \lvert u_{i+j+1}-u_{i+j} \rvert$ as a discretization of the
total variation  $TV_{\left[ x_{i},x_{i}+\delta\right] }\left( u\right)$.
We also define the following quantities where the function $g\left( s\right)$ is the {edge-stopping},
\[
\begin{array}{l}
\displaystyle
g_{i}^{k}=g\left( \dfrac{LV_{i}}{\varepsilon + TV_{i}} \right)\\
\\
\displaystyle
g_{i+\frac{1}{2}}^{k}=\dfrac{g_{i}^{k}+g_{i+1}^{k}}{2}\\
\\
\displaystyle
\partial_{x}u_{i}^{k+1} = \dfrac{u_{i+1}^{k+1}-u_{i}^{k+1}}{h} \\
\\
\displaystyle
\phi_{i+\frac{1}{2}}^{k}= g_{i+\frac{1}{2}}^{k} \partial_{x}u_{i}^{k+1}\\
\\
\displaystyle
\left[ \partial_{x}\left( g \partial_{x}u \right) \right] _{i}^{k} = \dfrac{\phi_{i+\frac{1}{2}}^{k}-\phi_{i-\frac{1}{2}}^{k}}{h} = \\
\\
\displaystyle
~~~~~~~=\dfrac{g_{i-1}^{k}+g_{i}^{k}}{2h^{2}}u_{i-1}^{k+1} - \dfrac{g_{i-1}^{k}+2g_{i}^{k}+g_{i+1}^{k}}{2h^{2}}u_{i}^{k+1} + \dfrac{g_{i}^{k}+g_{i+1}^{k}}{2h^{2}}u_{i+1}^{k+1} = \\
\\
\displaystyle
~~~~~~~= \beta_{i}^{k} u_{i-1}^{k+1} - \alpha_{i}^{k} u_{i}^{k+1} + \gamma_{i}^{k}u_{i+1}^{k+1}
\end{array}
\]
Then we can state our semi-implicit numerical scheme
\[
\dfrac{U^{k+1}-U^{k}}{\tau} = A\left( U^{k} \right) U^{k+1}
\]
where $U^k$ is the vector of the values $u_i^k$, and the matrix $A\left( U^{k} \right)$ is defined as
\begin{equation*}
A(U^k)=
\begin{pmatrix}
\alpha_{1}^{k} & \gamma_{1}^{k} & & \\
\beta_{2}^{k} & \alpha_{2}^{k} & \gamma_{2}^{k} &  \\
 & \ddots & \ddots & \ddots & \\
 & & \beta_{N-1}^{k} & \alpha_{N-1}^{k} & \gamma_{N-1}^{k} & \\
 & & & \beta_{N}^{k} & \alpha_{N}^{k}
\end{pmatrix}
\end{equation*}
Then at each time step we have to numerically solve the following linear system,
\[
\left( I-\tau A\left( U^{k} \right) \right) U^{k+1} = U^{k}
\]
Let $B\left( U^{k} \right)$ the matrix  $\left( I-\tau A\left( u^{k} \right) \right)$, it is easy to show that $B$
is a strictly diagonally dominant matrix, then it is non singular.
\subsection{$1D$ Test}
In this numerical experiment  we consider a recorded calcium imaging data from a $3D$ cultures of cortical neurons. The sampling rate was  $65$Hz and the sampling  time interval was about $8$ seconds. The data was collected in the Department of Neuroscience and Brain Technologies of the Fondazione Istituto Italiano di Tecnologia. In Figure \ref{fig:Rec_calcio}  we show a typical trace of the calcium signal together with different
smoothed signals at different time $T$.
\begin{figure}[htb]
\begin{center}
\includegraphics[width=0.35\textwidth]{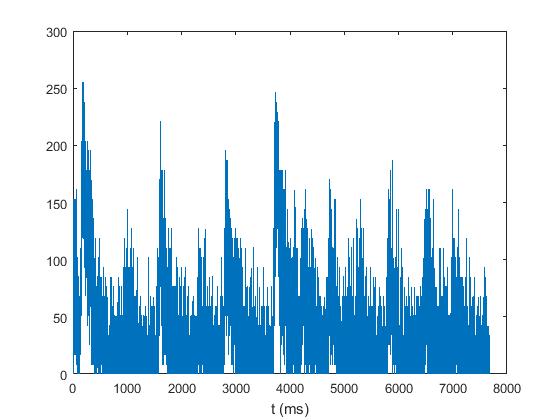}%
\includegraphics[width=0.35\textwidth]{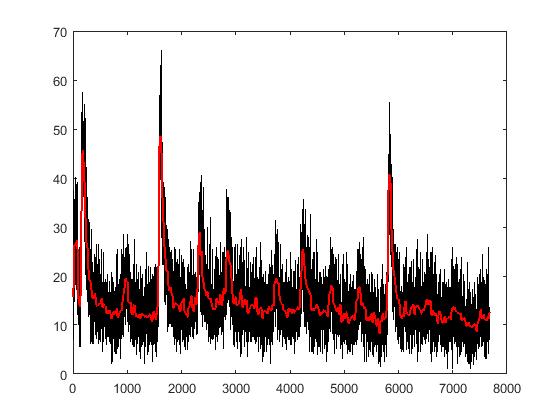}

\includegraphics[width=0.35\textwidth]{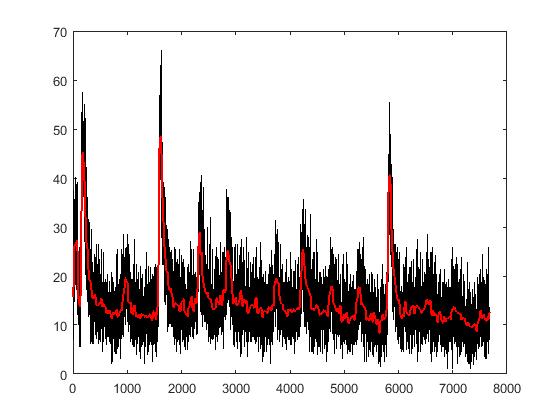}%
\includegraphics[width=0.35\textwidth]{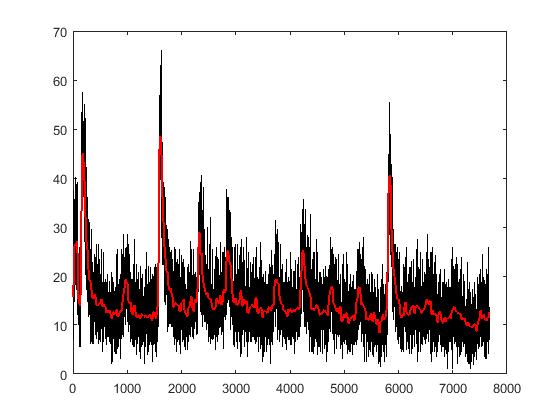} 
\end{center}
\caption{Reconstruction with the non linear, non local diffusion equation (\ref{eq_nonlocal_1D}) of a
Calcium trace. The original signal was obtained in the IIT lab based in Genova (Italy). First line, on the left the original signal, on the right
the solution for $T=300$. Second line, on the left the solution for $T=400$, on the right the solution for $T=500$. In all numerical experiments, 
$\delta =20$, $\tau =0.1$.} \label{fig:Rec_calcio}
\end{figure}
In the test we used $\tau =0.1$, $\delta =20$, and $h=1/65$, the initial data $U^0$ is obtained by the convolution of the original signal with a Gaussian filter with $\sigma=0.01$. 
\subsection{$2D$ problem}

Without loss of generality we can choose a square domain $(x,y)\in [1,L]\times [1,L]$, with $L$ a positive integer which is the number of pixels for each row (or column) of the image. Moreover, we can fix  the  grid spacing $h_x=h_y=1$ because it represents the distance between two adjacent pixels.
Also for the time interval $[0,T]$ we use an uniform grid with time spacing $\tau$, let $\mathbf{j} = (j_1,j_2)$, the node $(i,\mathbf{j})$
corresponds to the point $( i\tau , j_1 h_x, j_2 h_y )=( i\tau , j_1, j_2 )$.  We denote by $u_{i,(j_1+1,j_2)}$ the approximation of the solution $u$ in the 
node $(i,\mathbf{j})$.  Each  component of the gradient vector and the divergence of a vector will be approximated  with
central differences formula,
\begin{align*}
&\nabla u_{i,\mathbf{j}} =
\begin{pmatrix}
(u_{i,(j_1+1,j_2)}- u_{i,(j_1-1,j_2)})/2
\\
(u_{i,(j_1,j_2+1)}- u_{i,(j_1,j_2-1)})/2
\end{pmatrix}
, \\
&
\mathrm{div} (v_{i,\mathbf{j}},w_{i,\mathbf{j}}) =
\frac{v_{i,(j_1+1,j_2)}- v_{i,(j_1-1,j_2)}}{2}
+
\frac{w_{i,(j_1,j_2+1)}- w_{i,(j_1,j_2-1)}}{2}.
\end{align*}
The main effort, given a symmetric rectangular $Q= (-q_1,+q_1)\times (-q_2,+q_2)$,
is the computation of
\[
R_{Q,u_{i-1,\mathbf{j}}} = \frac{  \sup\{ \int_{x+Q} \nabla u (y,t) \nabla h\, dy, \ | \nabla h |_1 \leq 1 , h \text{ harmonic on }Q\}
 }{ \epsilon +  \int_{x+Q}
|\nabla u (y,t)|_1 dy }
.
\]
\noindent
\textbf{Step 1:} computation of $\int_{x+Q}
|\nabla u (y,t)|_1 dy $.
\begin{itemize}
\item We compute the total variation $TV_{i,(j_1+0.5,j_2+0.5)}$ in each square by taking the $\mathcal{Q}_1$ finite element  approximation $P(x_1,x_2)$
defined by its values at the corner nodes
$u_{i,(j_1,j_2)},u_{i,(j_1+1,j_2)},u_{i,(j_1+1,j_2+1)},u_{i,(j_1,j_2+1)}$. We may then exactly compute
\[
TV_{i,(j_1+0.5,j_2+0.5)} = \int_{0}^1\int_0^1
|\nabla P(x_1,x_2)|_1 dx_1\,dx_2
\]
as a function of the corner values;
\item with a pre-computed filter, we sum values of the vertices of each of the $4q_1q_2$ squares in the neighborhood
$Q$ of $\mathbf{j}$:
\[
TV_{Q,i,\mathbf{j}} = \sum_{k_1=j_1-q_1}^{q_1-1} \sum_{k_2=j_2-q_2}^{q_2-1}
TV_{i,(k_1+0.5,k_2+0.5)} .
\]
\end{itemize}
\noindent
\textbf{Step 2:} computation of
\begin{equation}\label{eq:numer_TBC}
\sup\{ \int_{x+Q} \nabla u (y,t) \nabla h\, dy, \ | \nabla h |_1 \leq 1 , h \text{ harmonic on }Q\} .
\end{equation}
Since $\mathrm{div} (\nabla h ) =0 $, in this case we have
\[
\int_{x+Q} \nabla u (y,t) \nabla h\, dy = \int_{x+\partial Q} u(y,t)  (\nabla h \cdot \vec{n}) \, ds ,
\]
and hence \eqref{eq:numer_TBC} depends on $u$ only through its values at $x+\partial Q$ (as in $1D$).
To compute
\[
\sup\{ \int_{x+\partial Q} u(s,t)  (\nabla h \cdot \vec{n}) \, ds, \ | \nabla h |_1 \leq 1 , h \text{ harmonic on }Q\} ,
\]
we approximate the set $\{h \text{ harmonic on }Q, | \nabla h |_1 \leq 1\}$ with a spectral decomposition using suitable eigenfunctions
in the following way:
\begin{itemize}
\item we map $Q$ on the square $S=[0,1]\times[0,1]$; up to a constant,
the following functions form a base for $\{h \text{ harmonic on } S\}$
\begin{align*}
&x, y, xy, \\
&
\frac{\sin(k \pi x) \sinh(k \pi ( 1 - y))}{\cosh(\pi k)} ,
\frac{\sin(k \pi y) \sinh(k \pi x)}{\cosh(\pi k)},\\
&
\frac{\sin(k \pi x) \sinh(k \pi y)}{\cosh(\pi k)},
\frac{\sin(k \pi y) \sinh(k \pi ( 1 - x))}{\cosh(\pi k)};
\end{align*}
\item we choose $M$ and we approximate the base by taking $k\leq M$,
the first $M$ modes and $x, y, xy$;
\item we compute the approximate base $\mathcal{H}= \{f_h, h=1,\ldots 4M+3\}$ on $Q$, with a linear transformation of $x$ e $y$ that preserves harmonicity;
\item we compute the gradient $\nabla f_h$ of each element of $\mathcal{H}$;
\item we solve the linear problem
\begin{multline*}
\max_{a_1,\ldots a_{4M+3}}  \sum_{h=1}^{4M+3} a_h \int_{x+\partial Q} u(s,t)  (\nabla f_h \cdot \vec{n}) \, ds
\\
\text{subject to }  \Big| \sum_{h=1}^{4M+3} a_h \nabla f_h  \Big|_1 \leq 1 , \text{ uniformly on $Q$},
\end{multline*}
by noticing that each component of $\sum_{h=1}^{4M+3} a_h \nabla f_h $ is still an harmonic function, and hence it attains its
maximum on $\partial Q$. With a very fine mesh $\{y_1, \ldots, y_L\}$ on the boundary $\partial Q$, the problem
\eqref{eq:numer_TBC} is hence well
approximated by the linear problem
\begin{align*}
\max_{a_1,\ldots a_{4M+3}}
&
\sum_{h=1}^{4M+3} a_h
\int_{x+\partial Q} u(s,t)  (\nabla f_h \cdot \vec{n}) \, ds
\\
&\text{subject to }
\sum_{h=1}^{4M+3} a_h \big( \pm \tfrac{\partial f_h (y_l) }{\partial x} \pm \tfrac{\partial f_h (y_l)}{\partial y} \big) \leq 1 ,
\\
&\qquad \text{ for any $l = 1,\ldots, L$}.
\end{align*}
This linear problem has $4M+3$ unknown and $4 L$ constrains and it
is solved independently for each node $(i,\mathbf{j})$. We have parallelized
this operation, once we have preallocated the quantities involving $\nabla f_h(y_l)$;
\item furthermore,
$u(s,t)$ is evaluated with a linear interpolation on the values $u_{i,\mathbf{j}} $
of the nodes $\mathbf{j} \in \partial Q$,
since the integral is made on the boundary of an element of
$\mathcal{Q}_1$. Therefore, there exist suitable constants such that
\[
\int_{x+\partial Q} u(y,t)  (\nabla f_h \cdot \vec{n}) \, ds =
\sum_{\mathbf{j} \in \partial Q} u_{i,\mathbf{j}} H_{h,\mathbf{j}}.
\]
Again, we have calculated the quantities $H_{h,\mathbf{j}}$ at the beginning of the code.
\end{itemize}
\begin{figure}[tbhp]
\begin{center}
\includegraphics[width=0.45\textwidth]{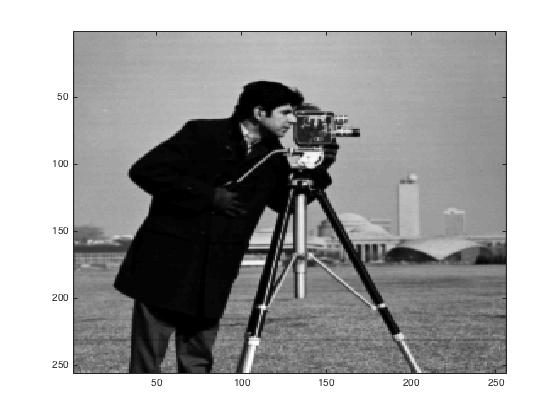}
\includegraphics[width=0.45\textwidth]{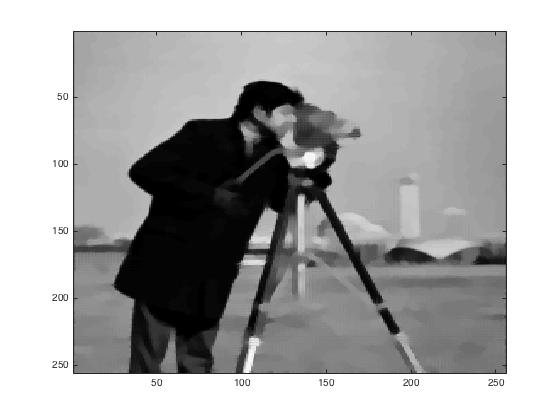}
\\
\includegraphics[width=0.45\textwidth]{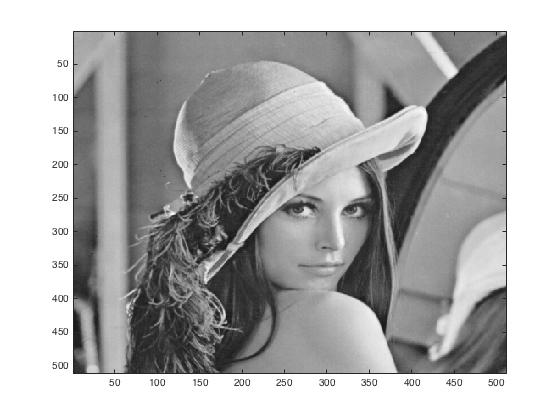}
\includegraphics[width=0.45\textwidth]{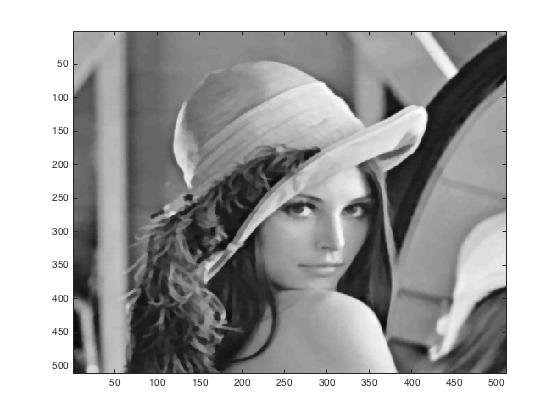}
\end{center}
\caption{Numerical solutions of equation \eqref{eq.linear}: the left
  plot refers to initial value, the right one is
  obtained after $300$ iterations. Resolution: top $126 \times 126 $ pixels, bottom $512 \times 512 $ pixels.
$Q = (-2,+2)^2$.}
\label{Fig:2D_photo-lena}
\end{figure}
In Figure~\ref{Fig:2D_photo-lena} two numerical simulations are performed. In order to avoid the solution of a linear system, as in the $1D$ case, we used  to advance in time the classical forward Euler scheme. We have therefore experimentally found a condition of stability by selecting an appropriate value
for the time spacing $\tau$. In both the cases, we have used $q_1=q_2=2$, $M=3$, $L=400$ with
Dual-Simplex Algorithm. The details under the magnitude of $Q$ are treated as noise,
and smoothed, while the edges are clearly magnified.
\section{Conclusion}
Image denoising/smoothing is one of basic issues in image processing. It plays a key preliminary step in many computer based
vision systems, but also it is  a starting point towards more  complex tasks.
Since image noise removal represents a relevant issue in various image analysis and computer
vision problems, it is a challenge to preserve the essential image features, such as edges and other
sharp structures during the smoothing process. The feature preserving image noise reduction still represents a challenging image
processing task. In this paper we propose a new method based on a non linear and non local diffusion equation.
The new approach has already been successfully applied in the analysis of membrane potentials in a neural network \cite{ANN15}, and for 
data analysis of recorded calcium signals in a $3D$ culture cells \cite{Difato}: these  denoising experiments  provided very encouraging results.
Here we focused on the mathematical analysis of the model and its numerical approximation. In particular, we provided an existence theorem for the variational solution and a numerical scheme both for the $1D$ and $2D$ case. \\
We observe that the uniqueness of the solution of the novel equation remains an open problem. Also the analysis of the stability and convergence of proposed numerical schemes should be completed. We have already developed some preliminary results that will be reported in a forthcoming paper. 
In particular, it is possible to show that the semi-implicit numerical scheme  satisfies the same discrete scale-space properties
as for  the Perona-Malik method. Finally, a more complete comparison with other methods has to be done, but this goes beyond the aims of this paper.
\section*{Acknowledgments}
The authors  are also extremely grateful to T. Nieus (UniMI) for providing them with
data of simulated membrane potentials, and to F. Difato (IIT-Ge) for providing the recorded calcium signals.


\begin{thebibliography}{10}

\bibitem{ANN15}
G.~Aletti,  G.~Naldi, and T.~Nieus,
\newblock{\it  From dynamics to links: a sparse reconstruction of the topology of a neural network},
\newblock {\tt arXiv:1501.06031 [stat.AP]}, 2015.


\bibitem{AGLM93}
L.~Alvarez, F.~Guichard, P.L.~Lions, and J.M.~Morel,
\newblock{\it Axioms and fundamental equations of image processing},
\newblock{Archive for rational mechanics and analysis},
Vol. 123(3) (1993), pp 199–-257.

\bibitem{APT2006}
S.~Angenent, E.~Pichon, and  A.~Tannenbaum,
{\it Mathematical Methods in Medical Image Processing},
Bulletin of the AMS, Vol. 43 (3) (2006), pp. 365--396.

\bibitem{AK2006}
G.~Aubert, and  P.~Kornprobst, {\it Mathematical problems in image processing: partial differential equations and the calculus of variations}, Springer-Verlag New York NY, 2006.

\bibitem{Brezis2010} H. Brezis,
\newblock {\it Functional Analysis, Sobolev Spaces and Partial Differential Equations},
\newblock Springer, New York NY,  2010.

\bibitem{CLMC1992}
{F.~Catt{\'e}, P.-L.~Lions, J.-M.~Morel, T.~Coll},
\newblock{\it Image Selective Smoothing and Edge Detection by Nonlinear Diffusion},
\newblock{SIAM J. Numer. Anal.}, {29} (1992), pp. 182--193.

\bibitem{CD2009}
A.~Chambolle, J.~Darbon, \textit{On Total Variation Minimization and Surface Evolution using Parametric
Maximum Flows}, International Journal of Computer Vision, vol. 84 n. 3 (2009), pp. 288–-307.

\bibitem{EG1992}
L.~Evans, R.F.~Gariepy,
\newblock{\it Measure Theory and Fine Properties of Functions},
\newblock Studies in Advanced Mathematics.
\newblock CRC Press, Boca Raton, FL, 1992.

\bibitem{Hum86}
R.A.~Hummel, \textit{Representations based on zero-crossings in scale-space}. In: Proc. IEEE CVPR, (1986)
pp. 204-209.

\bibitem{Kacur85}
J. Ka{\v{c}}ur,
\newblock {\it Method of {R}othe in evolution equations},
\newblock Teubner-Texte zur Mathematik Vol. 80, Leipzig, 1985.

\bibitem{KM1995}
{J.~Ka{\v{c}}ur, K.~Mikula},
\newblock{\it Solution of nonlinear diffusion appearing in image smoothing and
   edge detection},
\newblock{Appl. Numer. Math.}, {17} (1995), pp. 47--59.

\bibitem{Koe84}
J.J.~Koenderink, \textit{The structure of images}, Biol. Cypern., 50 (1984), pp. 363-370.

\bibitem{McO1996}
R.C.~McOwen,
\newblock {\it Partial Differential Equations, Methods and Applications},
\newblock Prentice-Hall, Upper Saddle River NJ, 1996.


\bibitem{Difato}
G.~Palazzolo, M.~Moroni, A.~Soloperto, G.~Aletti, G.~Naldi, M.~Vassalli, T.~Nieus, and F.~Difato,
\textit{Fast wide volume functional imaging of engineered in vitro brain tissues}, submitted. 


\bibitem{PM90}
P.~Perona, and J.~Malik,
\newblock{\it Scale-space and edge detection using anisotropic diffusion},
\newblock{IEEE Trans. Pattern. Anal. Machine Intell.}, 12 (1990), pp. 629--639.

\bibitem{RT2004}
P.A.~Raviart, J.M.~Thomas,
\newblock {\it Introduction \`{a} l'analyse num\'{e}rique des \'{e}quations aux d\'{e}riv\'{e}es partielles},
\newblock Sciences Sup, Dunod,  2004.

\bibitem{GS2001}
G.~Sapiro, {\it Geometric Partial Differential Equations and Image Analysis}, Cambridge University Press, 2001.

\bibitem{Wei98}
J.~Weickert,
\newblock {\it Anisotropic Diffusion in Image Processing},
\newblock Teubner, Stuttgart, 1998.

\bibitem{Wit83}
A.P.~Witkin, \textit{Scale-space filtering}, in: Proc. Eight Internat. Conf. on Artificial Intelligence, Vol. 2 (1983) pp. 1019-1022. 


\bibitem{Zie1989}
W.P.~Ziemer,
\newblock {\it Weakly Differentiable Functions},
\newblock Springer-Verlag, New York, 1989.

\end{thebibliography}
          \end{document}